\documentclass[reqno, 10pt]{amsart}
\usepackage{amssymb,amsmath,amsthm,amsfonts,color} 
\usepackage{etoolbox}

\makeatletter
\let\old@tocline\@tocline
\let\section@tocline\@tocline
\newcommand{\subsection@dotsep}{4.5}
\newcommand{\subsubsection@dotsep}{4.5}
\patchcmd{\@tocline}
  {\hfil}
  {\nobreak
     \leaders\hbox{$\m@th
        \mkern \subsection@dotsep mu\hbox{.}\mkern \subsection@dotsep mu$}\hfill
     \nobreak}{}{}
\let\subsection@tocline\@tocline
\let\@tocline\old@tocline

\patchcmd{\@tocline}
  {\hfil}
  {\nobreak
     \leaders\hbox{$\m@th
        \mkern \subsubsection@dotsep mu\hbox{.}\mkern \subsubsection@dotsep mu$}\hfill
     \nobreak}{}{}
\let\subsubsection@tocline\@tocline
\let\@tocline\old@tocline

\let\old@l@subsection\l@subsection
\let\old@l@subsubsection\l@subsubsection

\def\@tocwriteb#1#2#3{%
  \begingroup
    \@xp\def\csname #2@tocline\endcsname##1##2##3##4##5##6{%
      \ifnum##1>\c@tocdepth
      \else \sbox\z@{##5\let\indentlabel\@tochangmeasure##6}\fi}%
    \csname l@#2\endcsname{#1{\csname#2name\endcsname}{\@secnumber}{}}%
  \endgroup
  \addcontentsline{toc}{#2}%
    {\protect#1{\csname#2name\endcsname}{\@secnumber}{#3}}}%

\newlength{\@tocsectionindent}
\newlength{\@tocsubsectionindent}
\newlength{\@tocsubsubsectionindent}
\newlength{\@tocsectionnumwidth}
\newlength{\@tocsubsectionnumwidth}
\newlength{\@tocsubsubsectionnumwidth}
\newcommand{\settocsectionnumwidth}[1]{\setlength{\@tocsectionnumwidth}{#1}}
\newcommand{\settocsubsectionnumwidth}[1]{\setlength{\@tocsubsectionnumwidth}{#1}}
\newcommand{\settocsubsubsectionnumwidth}[1]{\setlength{\@tocsubsubsectionnumwidth}{#1}}
\newcommand{\settocsectionindent}[1]{\setlength{\@tocsectionindent}{#1}}
\newcommand{\settocsubsectionindent}[1]{\setlength{\@tocsubsectionindent}{#1}}
\newcommand{\settocsubsubsectionindent}[1]{\setlength{\@tocsubsubsectionindent}{#1}}

\renewcommand{\l@section}{\section@tocline{1}{\@tocsectionvskip}{\@tocsectionindent}{}{\@tocsectionformat}}%
\renewcommand{\l@subsection}{\subsection@tocline{1}{\@tocsubsectionvskip}{\@tocsubsectionindent}{}{\@tocsubsectionformat}}%
\renewcommand{\l@subsubsection}{\subsubsection@tocline{1}{\@tocsubsubsectionvskip}{\@tocsubsubsectionindent}{}{\@tocsubsubsectionformat}}%
\newcommand{\@tocsectionformat}{}
\newcommand{\@tocsubsectionformat}{}
\newcommand{\@tocsubsubsectionformat}{}
\expandafter\def\csname toc@1format\endcsname{\@tocsectionformat}
\expandafter\def\csname toc@2format\endcsname{\@tocsubsectionformat}
\expandafter\def\csname toc@3format\endcsname{\@tocsubsubsectionformat}
\newcommand{\settocsectionformat}[1]{\renewcommand{\@tocsectionformat}{#1}}
\newcommand{\settocsubsectionformat}[1]{\renewcommand{\@tocsubsectionformat}{#1}}
\newcommand{\settocsubsubsectionformat}[1]{\renewcommand{\@tocsubsubsectionformat}{#1}}
\newlength{\@tocsectionvskip}
\newcommand{\settocsectionvskip}[1]{\setlength{\@tocsectionvskip}{#1}}
\newlength{\@tocsubsectionvskip}
\newcommand{\settocsubsectionvskip}[1]{\setlength{\@tocsubsectionvskip}{#1}}
\newlength{\@tocsubsubsectionvskip}
\newcommand{\settocsubsubsectionvskip}[1]{\setlength{\@tocsubsubsectionvskip}{#1}}

\patchcmd{\tocsection}{\indentlabel}{\makebox[\@tocsectionnumwidth][l]}{}{}
\patchcmd{\tocsubsection}{\indentlabel}{\makebox[\@tocsubsectionnumwidth][l]}{}{}
\patchcmd{\tocsubsubsection}{\indentlabel}{\makebox[\@tocsubsubsectionnumwidth][l]}{}{}

\newcommand{\@sectypepnumformat}{}
\renewcommand{\contentsline}[1]{%
  \expandafter\let\expandafter\@sectypepnumformat\csname @toc#1pnumformat\endcsname%
  \csname l@#1\endcsname}
\newcommand{\@tocsectionpnumformat}{}
\newcommand{\@tocsubsectionpnumformat}{}
\newcommand{\@tocsubsubsectionpnumformat}{}
\newcommand{\setsectionpnumformat}[1]{\renewcommand{\@tocsectionpnumformat}{#1}}
\newcommand{\setsubsectionpnumformat}[1]{\renewcommand{\@tocsubsectionpnumformat}{#1}}
\newcommand{\setsubsubsectionpnumformat}[1]{\renewcommand{\@tocsubsubsectionpnumformat}{#1}}
\renewcommand{\@tocpagenum}[1]{%
  \hfill {\mdseries\@sectypepnumformat #1}}

\let\oldappendix\appendix
\renewcommand{\appendix}{%
  \leavevmode\oldappendix%
  \addtocontents{toc}{%
    \protect\settowidth{\protect\@tocsectionnumwidth}{\protect\@tocsectionformat\sectionname\space}%
    \protect\addtolength{\protect\@tocsectionnumwidth}{2em}}%
}
\makeatother

\makeatletter
\settocsectionnumwidth{2em}
\settocsubsectionnumwidth{2.5em}
\settocsubsubsectionnumwidth{3em}
\settocsectionindent{1pc}%
\settocsubsectionindent{\dimexpr\@tocsectionindent+\@tocsectionnumwidth}%
\settocsubsubsectionindent{\dimexpr\@tocsubsectionindent+\@tocsubsectionnumwidth}%
\makeatother

\settocsectionvskip{10pt}
\settocsubsectionvskip{0pt}
\settocsubsubsectionvskip{0pt}

\settocsectionformat{\bfseries}
\settocsubsectionformat{\mdseries}
\settocsubsubsectionformat{\mdseries}
\setsectionpnumformat{\bfseries}
\setsubsectionpnumformat{\mdseries}
\setsubsubsectionpnumformat{\mdseries}

\let\oldtableofcontents\tableofcontents
\renewcommand{\tableofcontents}{%
  \vspace*{-\linespacing}
  \oldtableofcontents}

\usepackage{mathrsfs,dsfont, comment,mathscinet}

\usepackage{enumerate,esint,accents}
\usepackage{mathtools}

\usepackage{graphicx}

\usepackage{epstopdf}
\usepackage[colorlinks=true, pdfstartview=FitV, linkcolor=blue, 
            citecolor=blue, urlcolor=blue]{hyperref}

\allowdisplaybreaks
\numberwithin{equation}{section}
\theoremstyle{plain}
\newtheorem{theorem}{Theorem}[section]
\newtheorem{proposition}[theorem]{Proposition}
\newtheorem{lemma}[theorem]{Lemma}

\theoremstyle{definition}
\newtheorem{definition}[theorem]{Definition}

\usepackage{amssymb,amsmath}

\newcommand\R{\mathbb R}
\def\Rz {\mathbb{R}}
\newcommand\M{\mathbb M}
\newcommand\N{\mathbb N}

\newcommand\ep{\varepsilon}

\newcommand\wk{\rightharpoonup}

\newcommand{\C}{\mathbb{C}}

\newcommand\be[1]{\begin{equation}\label{#1}}
\newcommand\ee{\end{equation}}
\newcommand\ba[1]{\begin{align}\label{#1}}
\newcommand\ea{\end{align}}
\newcommand\bas{\begin{align*}}
\newcommand\eas{\end{align*}}
\newcommand\nn{\nonumber}
\newcommand\mtwo{\mathbb{M}^{2\times 2}}
\newcommand\HH{\mathcal{H}}
\newcommand\EEE{\color{black}}

\newcommand \MMM{\color{black}}

\newcommand\Var{{\rm Var}}

\title[Derivation of a heteroepitaxial thin-film model]{Derivation of a heteroepitaxial thin-film model}
{}
\author[E. Davoli] {Elisa Davoli} 
\address[Elisa Davoli]{Department of Mathematics\\ University of Vienna\\Oskar-Morgenstern Platz 1\\1090 Vienna (Austria)}
\email[E. Davoli]{elisa.davoli@univie.ac.at}
\author[P.Piovano] {Paolo Piovano} 
\address[Paolo Piovano]{Department of Mathematics\\ University of Vienna\\Oskar-Morgenstern Platz 1\\1090 Vienna (Austria)}
\email[P. Piovano]{paolo.piovano@univie.ac.at}

\begin{document} 
\vskip .2truecm
\begin{abstract}
\small{\MMM A variational model for epitaxially-strained thin films on \MMM rigid \EEE substrates is derived both by $\Gamma$-convergence from a transition-layer setting, and by relaxation from a sharp-interface description available in the literature for regular configurations. 
The model is characterized by a configurational energy that accounts for both the competing mechanisms responsible for the film shape. On the one hand, the lattice mismatch between the film and the substrate generate large stresses, and corrugations may be present because film atoms move to release the elastic energy. On the other hand, flatter profiles may be preferable to minimize the surface energy. Some first regularity results are presented for energetically-optimal film profiles. \EEE}
\end{abstract}
\maketitle

\section{Introduction}

\MMM In this paper we introduce a variational model for describing heteroepitaxial growth of thin-films on a rigid substrate.\EEE


The first rigorous validation of a thin-film model as $\Gamma$-limit of \MMM a \EEE transition-layer model \MMM introduced in \EEE \cite{S2} was performed in the seminal paper \cite{FFLM2}.   Our analysis moves ahead from  \cite{FFLM2}, as \MMM in our energy functionals \EEE we not only consider the \MMM surface-energy contribution due to the \EEE free profile of the film and the surface of the substrate, but also \MMM that related to \EEE the interface between the film and the substrate, and we take into account the (possible) different elastic properties of the film and substrate materials. This is particularly important to fully treat the often encountered situation of \emph{heteroepitaxy}, i.e., the deposition of a material different from the one of the substrate. 

In order to describe our model, we need to introduce some notation. Following \cite{S2} we regard the substrate and the film as continua, we work in the framework of the \emph{theory of small deformations} \MMM in \EEE \emph{linear elasticity}, and, as in \cite{FFLM2}, we restrict our analysis to two-dimensional profiles (or three-dimensional configurations with planar symmetry). The interface between the film and the substrate is always assumed to be contained in the $x$-axis and the film thickness is measured by the height function $h:[a,b]\to[0,+\infty)$ with $b>a>0$. The subgraph $$\Omega_h:=\{(x,y):\,a<x<b,\,y<h(x)\},$$ is the region occupied by the film and the substrate material, whereas the \emph{graph} 
$$\Gamma_h:=\partial\Omega_h\cap\left((a,b)\times\R\right)$$
 of the height function $h$ represents the film profile. The elastic deformations of the film are encoded by the \emph{material displacement} $u:\Omega_h\to \R^2$, and its associated \emph{strain-tensor}, i.e., the symmetric part of the gradient of $u$, denoted by 
$$Eu:={{\rm sym}}\nabla u.$$
In order to account for non-regular profiles, as in \cite{FFLM2} the height function is assumed to be lower semicontinuous and with bounded pointwise variation. We denote by 
$${\tilde{\Gamma}_h}:=\partial\overline{\Omega}_h\cap\left((a,b)\times\R\right),$$
and by $\Gamma_h^{cut}$ the set of \emph{cuts} in the profile of $h$, namely $\Gamma_h^{cut}:={\Gamma}_h\setminus \tilde{\Gamma}_h$.

As previously mentioned, elasticity must be included in the model as it plays a major role in heteroepitaxy. Large stresses are in fact induced  in the film by the lattice mismatch between the film and the substrate materials \cite{FG}. We introduce a parameter $e_0\geq 0$ to represent such lattice mismatch and, as in \cite{FFLM2}, we assume that the minimum of the energy is reached at
$$E_0(y):=\begin{cases}e_0\, (\bf{e_1}\odot{\bf{e_1}})&\text{if }y\geq 0\\
0&\text{otherwise},
\end{cases}$$
where $({\bf e_1}, {\bf e_2})$ is the standard basis of $\R^2$. In the following we refer to $E_0$ as the \emph{mismatch strain}.  

The model considered in this paper is characterized by an energy functional $\mathcal{F}$, defined for any film configuration $(u,h)$ as
\begin{align}\label{filmenergyfinal}
 \mathcal{F}(u,h)\,=\,\int_{\Omega_h}W_{0}(y,&Eu(x,y)-E_0(y))\,dx\,dy\nonumber\\
&\qquad+\,\int_{\tilde{\Gamma}_h}\varphi(y)\,d\HH^1\,+\, \gamma_{\rm fs}(b-a)\,+\,2\gamma_{\rm f} \HH^1(\Gamma_h^{cut}),
\end{align}
 where the surface density $\varphi$ is given by
$$\varphi(y):=\begin{cases}\gamma_f&\text{if }y> 0,\\
 \min\{\gamma_f, \gamma_s-\gamma_{fs}\}&\text{otherwise,}\end{cases}$$
with 
\be{eq:ass-gammas}
\gamma_f>0,\quad \gamma_s>0,\quad\text{and}\quad\gamma_s-\gamma_{fs}\geq 0.
\ee
The elastic energy density \mbox{$W_0:\R\times \mtwo_{\rm sym}\to [0,+\infty)$} satisfies 
$$W_0(y,E):=\frac12 E:\C(y)E$$
for every $(y,E)\in \R\times \mtwo_{\rm sym}$. 
In the expression above $\C(y)$ represents the elasticity tensor, 
$$\C(y):=\begin{cases} \C_{f}&\text{if }y>0,\\
\C_s&\text{otherwise},
\end{cases}$$
and is assumed to satisfy
\be{eq:positive definite}E:\C(y)E>0\ee
for every $y\in\R$ and $E\in\mtwo_{\rm sym}$.
 The fourth-order tensors $\C_f$ and $\C_s$ are \MMM symmetric, positive-definite, and possibly different. \EEE Our model includes therefore the case of a different elastic behavior \MMM of \EEE the film and the substrate. 

Energy functionals of the form \eqref{filmenergyfinal} appear in the study of \emph{Stress-Driven Rearrangement Instabilities  (SDRI)} \cite{Gr} and well represent the \emph{competition} between the \emph{roughening effect} of the elastic energy and the \emph{regularizing effect} of the surface energy that characterize the formation of such crystal microstructures (see \cite{FG,Ga,Gr} and \cite{FFLMi} for the related problem of crystal cavities). 

As already mentioned at the beginning of the introduction a similar functional to \eqref{filmenergyfinal} was derived  in \cite{FFLM2}  by $\Gamma$-convergence from the transition-layer model introduced in \cite{S2} in the case in which $\C_f=\C_s$, and $\gamma_{fs}=0$. We observe here that in \cite{FFLM2} the regularity of the local minimizers of such energy is studied for isotropic film and substrate in the case in which $\gamma_f\leq\gamma_s$, and the local minimizers  are shown to be smooth outside of finitely many \emph{cusps} and \emph{cuts} and to form zero contact angles with the substrate (see also \cite{BC,FFLMi}). In the same regime in \cite{FM} \emph{thresholds} for the film volume (dependent on the lattice mismatch), below which the flat configuration is an absolute minimizer or only a local minimizer, and below which minimizers are smooth, have been identified \MMM (see also \cite{Bo1,Bo2} for the anisotropic setting)\EEE. We point out that the functional in \cite{FFLM2}, when restricted to the regime $\gamma_f\leq\gamma_s$ did not present any discontinuity along the film/substrate interface contained in the $x$-axis. The same applies for the energy in \cite{FM}. In our more general setting, instead, \eqref{filmenergyfinal}  always presents a sharp discontinuity with respect to the elastic tensors. Additionally the relaxation results of this paper include the \emph{dewetting regime}, $\gamma_f>\gamma_s-\gamma_{fs}$, for which the surface tension \MMM also presents a sharp discontinuity. \EEE 

Related SDRI models have been studied in \cite{BGZ,FPZ,GZ}. In \cite{FPZ} the existence and the shape of island profiles, which enforces the presence of nonzero contact angles, has been analyzed in the constraint of faceted profiles. In \cite{GZ} a mathematical justification of island nucleation was provided by deriving scaling laws for the minimal energy in terms of  $e_0$ and the film volume, and then extended in \cite{BGZ} to the situation of  unbounded domains, in the two regimes of small- and large-slope approximations for the profile function  $h$. Finally, the evolutionary problem for thin-film profiles has been studied in dimension two in \cite{FFLM} for the evolution driven by surface diffusion, and in \cite{P} for the growth in the evaporation-condensation case (see also \cite{DFL,FLL} for a related model describing vicinal surfaces in epitaxial growth). Recently the analysis of \cite{FFLM} has been extended to three dimensions in  \cite{FFLM3}. A complete analysis of the regularity of optimal profiles, as well as of contact angle conditions will be the subject of the companion paper \cite{davoli.piovano}.

The paper is organized as follows. \MMM In Section \ref{sec:mainresults} we introduce the mathematical setting and we rigorously state our main result (see Theorem \ref{thm:justification}). Section \ref{sec:justification} is devoted to the analytical derivation of the energy \eqref{filmenergyfinal} by relaxation and by $\Gamma$-convergence, respectively, from the sharp-interface and the transition-layer models. In  Section \ref{sec:regularity}  we present a first regularity analysis for the local minimizers of such energy. We first perform a volume penalization of the energy to allow more freedom in the admissible variations, and finally prove in our setting the \emph{internal-ball condition}, an idea first introduced in \cite{CL} and employed also in \cite{FFLMi,FFLM2}. \EEE

\section{\MMM Setting of the problem and \EEE main results}\label{sec:mainresults}

\subsection{Mathematical setting} \label{subset:mathset} In this subsection we introduce the main definitions and the notation used throughout the paper. 
We begin by characterizing the admissible film profiles. The set  $AP$  of admissible film profiles in $(a,b)$ is denoted by
$$AP(a,b):=\{h: [a,b]\to[0,+\infty)\,:\,\textrm{$h$ is lower semicontinuous and $\textrm{Var}\,h<+\infty$}\},$$
where $\textrm{Var}\,h$ denotes the pointwise variation of $h$, namely,
\begin{align*}
\textrm{Var}\,h :=\sup\Big\{  &\sum^n_{i=1}|h(x_i)-h(x_{i-1})|\,:\\
&\qquad \qquad\textrm{$P:=\{x_1,\dots,x_n\}$ is a partition of $[a,b]$} \Big\}.
\end{align*}
We recall that for every lower semicontinuous function $h: [a,b]\to[0,+\infty)$, to have finite pointwise variation is equivalent to the condition 
$$\HH^1(\Gamma_h)<+\infty,$$
where
$$\Gamma_h:=\partial\Omega_h\cap\left((a,b)\times\R\right).$$
 For every  $h\in AP(a,b)$, and for every $x\in (a,b)$, consider the left and right limits
$$ h(x^\pm):=\lim_{z\to x^{\pm}} h(z).$$
We define
$$h^-(x):=\min\{h(x^+),h(x^-)\}=\liminf_{z\to x} h(z),$$
and 
$$h^+(x):=\max\{h(x^+),h(x^-)\}=\limsup_{z\to x} h(z).$$
In the following  ${\rm Int}(A)$ denotes the interior part of a set $A$.
Let us now recall some properties of height functions $h\in AP(a,b)$, regarding their graphs $\Gamma_h$, their subgraphs $\Omega_h$, the film and the substrate parts of the subgraph, \MMM namely \EEE
$$\Omega_h^+:=\Omega_h\cap\{y>0\}$$
and
$$ \Omega_h^-:=\Omega_h\cap\{y\leq 0\}$$ respectively, and the set 
\be{eq:Gamma-tilda}\tilde{\Gamma}_h:=\partial\bar{\Omega}_h\cap ((a,b)\times \R).\ee
Any $h\in AP(a,b)$ satisfies the following assertions (see \cite[Lemma 2.1]{FFLM2}):
\begin{enumerate}
\item[1.] $\Omega_h^+$ has finite perimeter in $((a,b)\times \R)$,
\item[2.]  $\Gamma_h=\{(x,y):\,a<x<b,\,h(x)<y<h^+(x)\}$,
\item[3.]  $h^-$ is lower semicontinuous and ${\rm Int}\left(\overline{\Omega}\right)=\{(x,y):\,a<x<b,\, y<h^-(x)\}$,
\item[4.]  $\tilde{\Gamma}_h=\{(x,y):\,a<x<b,\,h^-(x)\leq y\leq h^+(x)\}$,
\item[5.]  $\Gamma_h$ and $\tilde{\Gamma}_h$ are connected.
\end{enumerate}

We now characterize various portions of $\Gamma_h$. To this aim we denote the \emph{jump} set of a function $h\in AP(a,b)$, i.e., the set of its profile discontinuities, by
\be{eq:def-S-h}J(h):=\{x\in (a,b):\,h^-(x)\neq h^+(x)\},\ee
whereas the set \MMM of \EEE \emph{vertical cuts} in the graph of $h$ is given by
\be{eq:def-S}C(h):=\{x\in (a,b):\,h(x)<h^-(x)\}.\ee
The graph $\Gamma_h$ of a height function $h$ is then characterized by the decomposition 
$$\Gamma_h=\Gamma_h^{jump}\sqcup\Gamma_h^{cut}\sqcup\Gamma_h^{graph},$$
where $\sqcup$ denotes the disjoint union, and 
\begin{align}
\nn\Gamma_h^{jump}&:=\overline{\{(x,y):\,x\in (a,b)\cap J(h),\, h^-(x)\leq y\leq h^+(x)\}},\\
\Gamma_h^{cut}&:=\{(x,y):\,x\in (a,b)\cap C(h),\, h(x)\leq y< h^-(x)\},\label{eq:cuts}\\
\nn\Gamma_h^{graph}&:=\Gamma_h\setminus(\Gamma_h^{jump}\cup\Gamma_h^{cut}). 
\end{align}

 \begin{figure}[htp]
\begin{center}
\includegraphics[width=12.5cm]{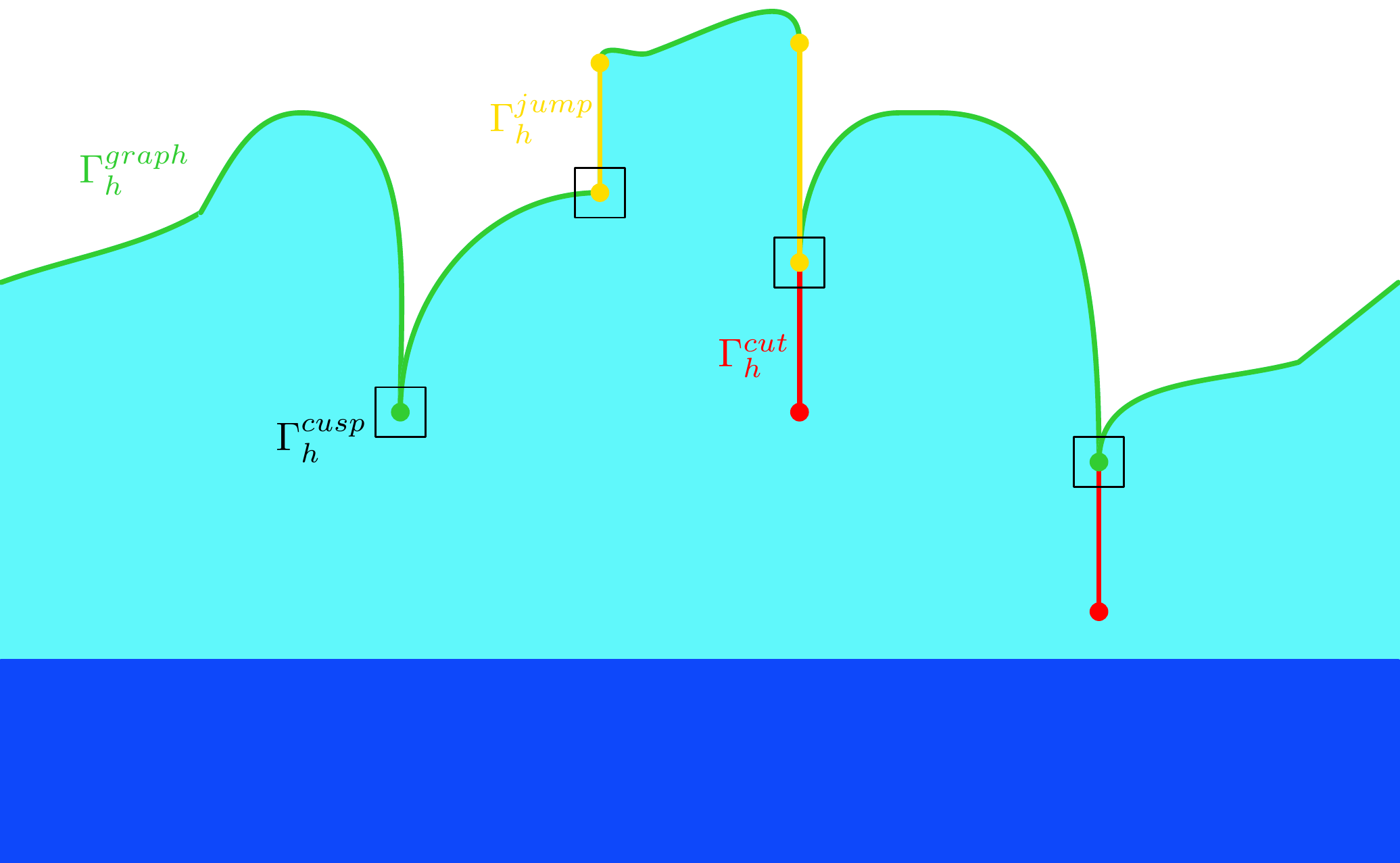}
\caption{In the figure above an admissible profile function $h$ is displayed. The portions of $\Gamma_h$ corresponding to $\Gamma_h^{graph}$, $\Gamma_h^{jump}$,  and $\Gamma_h^{cut}$  are represented with the colors green, yellow, and red respectively. The points in $\Gamma_h^{cusp}$ are marked by enclosing them within squares.}
\label{graphs}
\end{center}
\end{figure}
We observe that  $\Gamma_h^{graph}$ represents the regular part of the graph of $h$, whilst both $\Gamma_h^{jump}$ and $\Gamma_h^{cut}$ consist in (at most countable) unions of segments, corresponding to the \emph{jumps} and the \emph{cuts} in the graph of $h$, respectively (see Figure \ref{graphs}). Notice also that
$$\Gamma_h=\tilde{\Gamma}_h\sqcup \Gamma_h^{cut}.$$

\MMM Denoting by $h'_{-}(x)$ and  $h'_{+}(x)$ the left and right derivatives of $h$ in a point $x$, respectively, we \EEE identify  the set of \emph{cusps} in $\Gamma_h$ by
\begin{align*}
\Gamma_h^{cusp}:=\big\{ (x,h^-(x))\,:\,\,&\textrm{either $x\in J(h)$}\\
 &\textrm{or we have that $x\not\in J(h)$ with $h'_+(x)=+\infty$ or $h'_-(x)=-\infty$}\big\}
\end{align*}
(see Figure \ref{graphs}). 

For every $h\in AP(a,b)$ we indicate its set of of zeros by
$$Z_h:=\Gamma_h\cap\{x\in[a,b]\,:\,h(x)=0\}.$$

We now define the family $X$ of admissible film configurations as
$$X:=\{(u,h):\,\textrm{$u\in H^1_{\rm loc}(\Omega_h;\R^2)$ and $h\in AP(a,b)$}\}$$ 
and we endow $X$ with the following notion of convergence.

\begin{definition}
\label{def:conv-X}
We say that a sequence $\{(u_n,h_n)\}\subset X$ converges to $(u,h)\in X$, and we write $(u_n,h_n)\to (u,h)$ in $X$ if
\begin{enumerate}
\item[1.] $\sup_n \Var\, h_n<+\infty$,
\item[2.] $\R^2\setminus\Omega_{h_n}$ converges to $\R^2\setminus\Omega_h$ in the Hausdorff metric,
\item[3.] $u_n\wk u$ weakly in $H^1(\Omega';\R^2)$ for every $\Omega'\subset\subset\Omega_h$.
\end{enumerate}
\end{definition}
Let us also consider the following subfamily $X_{\rm Lip}$ of configurations with Lipschitz profiles, namely,
$$X_{\rm Lip}:=\{(u,h):\,u\in H^1_{\rm loc}(\Omega_h;\R^2),\,h\text{ is Lipschitz}\}.$$
We recall that the thin-film model analyzed in this paper is characterized by the energy $\mathcal{F}$ defined \MMM in \EEE \eqref{filmenergyfinal} \MMM and evaluated \EEE on configurations $(u,h)\in X$.

We state here the definition of \emph{$\mu$-local minimizers} of the energy $\mathcal{F}$. 
\begin{definition}
\label{def:local-min}
We say that a pair $(u,h)\in X$ is a $\mu$-local minimizer of the functional $\mathcal{F}$ if $\mathcal{F}(u,h)<+\infty$ and 
$$\mathcal{F}(u,h)\leq \mathcal{F}(v, g)$$
for every $(v,g)\in X$ satisfying $|\Omega_g^+|=|\Omega_h^+|$ and $|\Omega_g\Delta \Omega_h|\leq \mu$.
\end{definition}
\noindent Note that every global minimizer (with or without volume constraint) is a $\mu$-local minimizer.

\subsection{The sharp-interface and the  transition-layer models}\label{subsec:spenser models} We now recall classical thin-film models from the Literature. The sharp-interface model for epitaxy is characterized by a configurational energy  $\mathcal{F}_0$ that presents a discontinuous transition both in the elasticity tensors and in the surface tensions, and that encodes the abrupt change in materials across the film/substrate interface at the $x$-axis. We set 
\begin{align}\label{filmenergy}
\mathcal{F}_0(u,h):=\int_{\Omega_h}W_0(y, &Eu(x,y)-E_0(y))\,dx\,dy\nonumber\\
&\qquad+\int_{\Gamma_h}\varphi_0(y)\,d\HH^1\,+\, \gamma_{fs}\HH^1((a,b)\setminus Z_h)
\end{align}
for every $(u,h)\in X_{\rm Lip}$, where the energy density $\varphi_0:\R\to [0,+\infty)$ forces a sharp discontinuity at $\{y=0\}$, namely
$$\varphi_0(y):=\begin{cases}\gamma_f&\text{if }y> 0\\
\gamma_s&\text{if }y= 0,\end{cases}$$
for positive constants $\gamma_f$ and $\gamma_s$. The same energy functional has been considered in \cite{S2}, where it appears without the last term since in that framework $\gamma_{fs}$ is considered to be negligible. We notice that $\mathcal{F}$ and $\mathcal{F}_0$ differ only with respect to the surface energy, and that $\mathcal{F}$ is extended to the set $X$.

 Models presenting regularized discontinuities have been introduced in the Literature because more easy to implement numerically (see, e.g., \cite{S2}). They can be considered as an approximation of the sharp-interface functional $\mathcal{F}_0$ where 
the elastic tensors and/or the surface densities are regularized over a thin transition region of width $\delta > 0$ (see Figure \ref{filmdelta}). 
\begin{figure}[htp]
\begin{center}
\includegraphics[width=11cm]{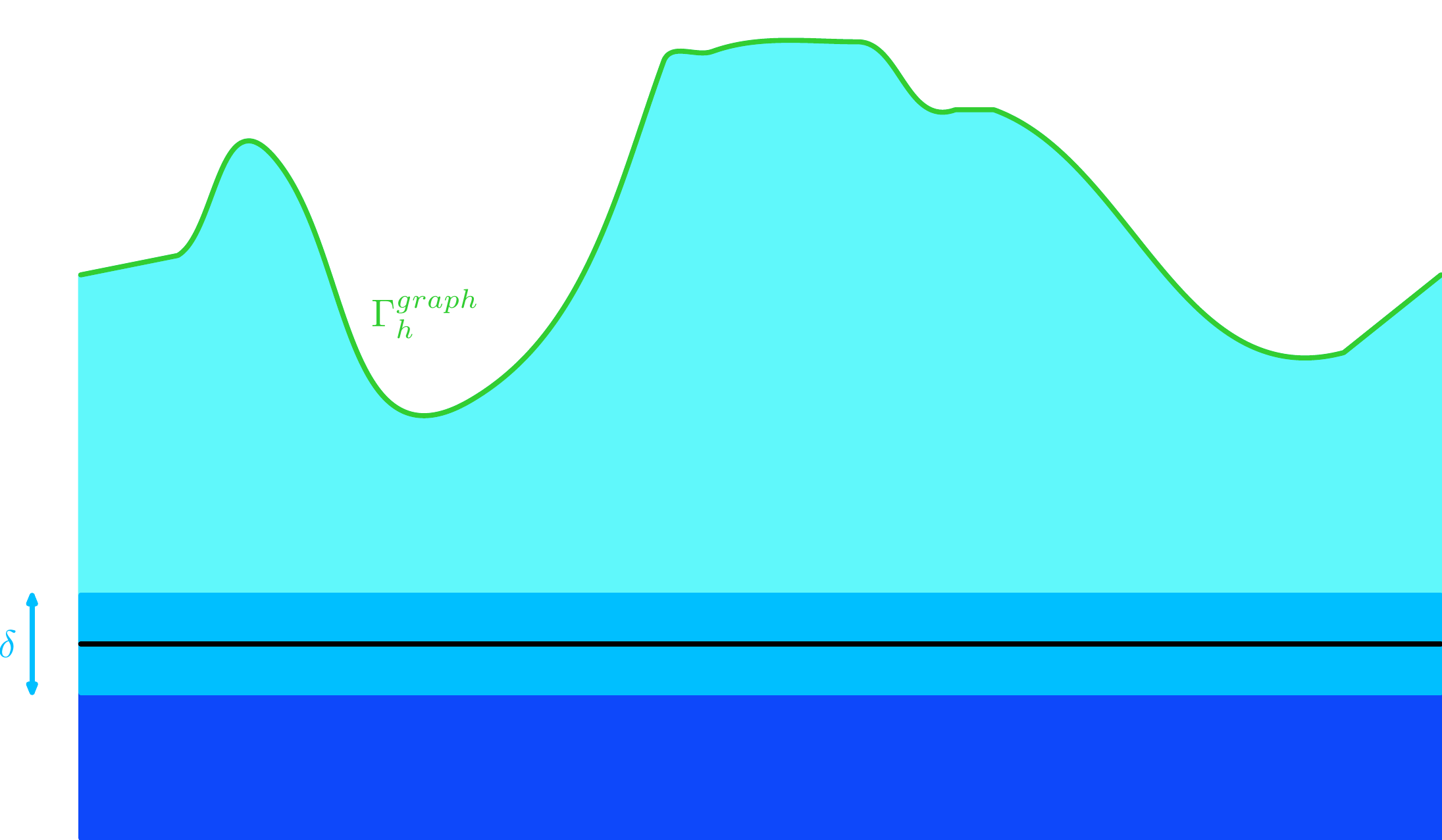}
\caption{In the transition-layer model the elastic tensors and the surface tension are regularized over a (thin) layer with thickness $\delta>0$.}
\label{filmdelta}
\end{center}
\end{figure}

In order to introduce the energy functional $\mathcal{F}_{\delta}$ corresponding to the \emph{transition-layer model} with transition layer of width $\delta > 0$, we consider an auxiliary smooth and increasing function $f$ such that $f(0)=0$, $\lim_{y\to +\infty}f(y)=1$, $\lim_{y\to -\infty}f(y)=-1$, and
\be{hypof}\int_{-\infty}^0(1+ f(y))^2\,dy<+\infty.\ee
We notice that the hypotheses on $f$ are satisfied for example by the \emph{boundary-layer function}
$$
r\mapsto\frac{2}{\pi}\arctan(r)
$$
proposed in \cite{KF,KF2} (see also \cite{S2}).  The regularized mismatch strain is defined as
$$E_{\delta}(y):=\frac12 e_0 \Big(1+f\Big(\frac{y}{\delta}\Big)\Big){\bf e_1}\odot{\bf e_1}\quad\text{for every }y\in\R,$$
whereas the regularized surface energy density takes the form
$$\varphi_{\delta}(y):=\gamma_f f\left(\frac{y}{\delta}\right)+ (\gamma_s-\gamma_{fs})\left(1- f\left(\frac{y}{\delta}\right)\right),$$ 
for every $y\in \R$ (see \cite{S}).

The transition-layer energy functional is then  given by
\begin{align*}
\mathcal{F}_{\delta}(u,h):=\int_{\Omega_h}W_{\delta}(y, Eu(x,y)-&E_{\delta}(y))\,dx\,dy\nonumber\\
&\qquad+\int_{\Gamma_h}\varphi_{\delta}(y)\,d\HH^1+\gamma_{fs}(b-a)
\end{align*}
for every $(u,h)\in X_{\rm Lip}$, where
$W_{\delta}(y,E):=\frac12 E:\C_{\delta}(y)E$ for every $y\in\R$ and $E\in\mtwo_{\rm sym}$, with 
\begin{align*}
\C_{\delta}(y):=\frac12\Big(1+f\Big(\frac{y}{\delta}\Big)\Big)\C_f&+\frac12\Big(1-f\Big(\frac{y}{\delta}\Big)\Big)\C_s\nonumber\\
&\qquad+\frac12\Big(1+f\Big(\frac{y}{\delta}\Big)\Big)\Big(1-f\Big(\frac{y}{\delta}\Big)\Big)(\C_s-\C_f).
\end{align*}
Notice that $\C_{\delta}(0)=\C_{s}$, and that $\C_{\delta}(y)$ is symmetric and positive-definite for every $y\in \R$. Additionally, there exists a positive constant $C$ such that
\be{eq:quadratic-form}
\C_{\delta}(y)F:F\leq C|F|^2\quad\text{for every }F\in \M^{2\times 2}.
\ee

\subsection{Statement of the main result} The main result of the paper concerns the derivation of the energy functional $\mathcal{F}$ from the transition-layer functional $\mathcal{F}_{\delta}$ and the sharp-interface model $\mathcal{F}_0$. 


\begin{theorem}[Model derivation]
\label{thm:justification}
The energy $\mathcal{F}$ is both 
\begin{enumerate}
\item[1.] The relaxed functional of $\mathcal{F}_0$, i.e., 
\bas&\mathcal{F}(u,h):=\inf\left\{\liminf_{n\to +\infty} \mathcal{F}_{0}(u_n,h_n):\,(u_n,h_n)\in X_{\rm Lip},\right.\\
&\qquad\qquad\qquad\qquad\qquad\,(u_n,h_n)\to (u, h)\,\text{in }X,\,\text{and }|\Omega_{h_n}^+|=|\Omega_h^+| \bigg\}
\end{align*}
for every $(u,h)\in X$.
\item[2.] The $\Gamma$-limit as $\delta\to 0$ of the transition layer energies $\mathcal{F}_{\delta}$ under the volume constraint. 
\end{enumerate}
\end{theorem}

\section{Derivation of the thin-film model}\label{sec:justification}
In this section we provide a rigorous justification of the model $\mathcal{F}$ defined in \eqref{filmenergyfinal} by proving Theorem \ref{thm:justification}.

 \begin{proof}[Proof of Theorem \ref{thm:justification}]
 Assertion 1. and 2. of Theorem \ref{thm:justification} follow, respectively, from Propositions \ref{prop:relax} and  \ref{prop:gamma-conv}, which are proven in the following two subsections.
 \end{proof}

\subsection{Relaxation from the sharp-interface model} 

In this subsection we characterize $\mathcal{F}$ as the lower-semicontinuous envelope of the energy $\mathcal{F}_0$ with respect to the convergence in $X$, restricted to pairs in $X_{\rm Lip}$. To this aim we begin with an auxiliary result that will be fundamental in the proof of Proposition \ref{prop:relax}.

\begin{lemma}\label{Vitali_lemma}
Let  $h_n\in L^1((a,b); [0,+\infty))$ \MMM be \EEE such that $h_n\to h$ in $L^1(a,b)$. For every sequence $\{\lambda_n\}$ converging to 0, there exist a constant $\mu>0$ (depending on the sequences $\{\lambda_n\}$, $\{h_n\}$, and \MMM on \EEE $h$) and an integer $N_{\mu}$ such that 
\begin{equation}\label{lowerbound}
|H_{\lambda_n}|\,+\,\frac{1}{\lambda_n}\int_{[a,b]\setminus H_{\lambda_n}}h_n(x_1)\,\mathrm{d}x_1>\mu
\end{equation}
for every $n\geq N_{\mu}$, where $H_{\lambda_n}:=\{x_1\in[a,b]\,:\, h_n(x_1)\geq\lambda_n\}$.
\end{lemma}
\begin{proof}
By contradiction, up to passing to a \MMM (not relabeled) subsequence \EEE both for $\{\lambda_n\}$ and $\{h_n\}$ we have that 
\begin{equation}\label{lowerbound_contradicted}
|H_{\lambda_n}|\,+\,\frac{1}{\lambda_n}\int_{[a,b]\setminus H_{\lambda_n}}h_n(x_1)\,\mathrm{d}x_1\leq \mu_n
\end{equation}
for some \MMM sequence $\{\mu_n\}$ \EEE converging to zero.
\MMM Fix \EEE $\eta\in(0,\|h\|_{L^1(a,b)})$. By Vitali's Theorem there exists $\mu_{\eta}>0$ such that 
$$
\|h_n\|_{L^1(S)}\leq \|h\|_{L^1(a,b)}-\eta
$$
for every measurable set $S$ with $|S|\leq\mu_{\eta}$ and $n\in \mathbb{N}$.  From \eqref{lowerbound_contradicted} it follows that 
$|H_{\lambda_n}|\leq\mu_{\eta}$ for $n$ large enough, and hence we obtain that
\begin{equation}\label{bound}
\|h_n\|_{L^1(H_{\lambda_n})}\leq \|h\|_{L^1(a,b)}-\eta
\end{equation}
\MMM for $n$ large enough. \EEE 
However, by \eqref{lowerbound_contradicted} we also have that 
\begin{align}
\label{eq:contr-new}
0\leftarrow\mu_n\geq \frac{1}{\lambda_n}\int_{[a,b]\setminus H_{\lambda_n}}h_n(x_1)\,\mathrm{d}x_1
&=\frac{1}{\lambda_n}\left(\|h_n\|_{L^1(a,b)}- \|h_n\|_{L^1(H_{\lambda_n})}\right)\\
&\geq  \frac{1}{\lambda_n}\left(\|h_n\|_{L^1(a,b)}-\|h\|_{L^1(a,b)})+\eta\right)\notag
\end{align}
where we used \eqref{bound} in the last inequality. \MMM Since $\lambda_n\to0$ and $h_n\to h$ in $L^1(a,b)$, there holds \EEE
$$
 \frac{1}{\lambda_n}\left(\|h_n\|_{L^1(a,b)}-\|h\|_{L^1(a,b)}+\eta\right)\to+\infty.
$$
\MMM This contradicts \eqref{eq:contr-new} and concludes the proof of the lemma. \EEE
\end{proof}

We are now ready to prove the main result of this subsection.

\EEE

\begin{proposition}[Relaxation of the sharp-interface model]
\label{prop:relax}
\begin{align*}
\mathcal{F}(u,h)&=\inf\left\{\liminf_{n\to +\infty} \mathcal{F}_{0}(u_n,h_n):\,(u_n,h_n)\in X_{\rm Lip},\right.\\
&\qquad\qquad\qquad\qquad\qquad\,(u_n,h_n)\to (u,h)\,\text{in }X,\,\text{and }|\Omega_{h_n}^+|=|\Omega_h^+| \bigg\},
\end{align*}
for every $(u,h)\in X$.
\end{proposition}
\begin{proof}
We preliminary observe that the thesis is equivalent to showing that
\begin{align}
\nonumber\bar{\mathcal{F}}(u,h)&:=\inf\left\{\liminf_{n\to +\infty} \tilde{\mathcal{F}}_{0}(u_n,h_n):\,(u_n,h_n)\in X_{\rm Lip},\right.\\
\label{eq:equiv-rel}&\qquad\qquad\quad\,(u_n,h_n)\to (u,h)\,\text{in }X,\,\text{and }|\Omega_{h_n}^+|=|\Omega_h^+| \bigg\}=\tilde{\mathcal{F}}(u,h)
\end{align}
for every $(u,h)\in X$, where
$$\tilde{\mathcal{F}}_{0}(u,h):=\int_{\Omega_h}W_0(y, Eu(x,y)-E_0(y))\,dx\,dy+\int_{\Gamma_h}\tilde{\varphi}_0(y)\,d\HH^1,$$
with
$$\tilde{\varphi}_0(y):=\begin{cases}\gamma_f&\text{if }y>0\\\gamma_s-\gamma_{fs}&\text{otherwise}\end{cases},$$
and
$$\tilde{\mathcal{F}}(u,h):=\mathcal{F}(u,h)-\gamma_{fs}(b-a).$$
The proof of the inequality
$$\bar{\mathcal{F}}(u,h)\geq \tilde{\mathcal{F}}(u,h)$$
for every $(u,h)\in X$ follows along the lines of \cite[Proof of Theorem 2.8, Step 1]{FFLM2}, by observing that
$$\liminf_{n\to+\infty}\int_{\Gamma_{h_n}}\MMM\tilde{\varphi_0}\EEE(y)\,d\HH^1\geq \liminf_{n\to+\infty}\int_{\Gamma_{h_n}}{\varphi}(y)\,d\HH^1,$$
and by applying the argument in \cite[(2.22)--(2.26)]{FFLM2} directly to the density $\varphi$, which is lower-semicontinuous and hence allows to use Reshetnyak's theorem (see \cite[Theorem 2.38]{AFP}).

Fix now $(u,h)\in X$. To prove that
$$\bar{\mathcal{F}}(u,h)\leq \tilde{\mathcal{F}}(u,h)$$
it is enough to construct a sequence $\{(u_n,h_n)\}\subset X_{\rm Lip}$ such that
\begin{align}
&(u_n,h_n)\to (u,h)\quad\text{in }X,\label{condition_1}\\
&|\Omega_{h_n}^+|=|\Omega_h^+|, \label{condition_2}
\end{align}
and
\be{limsup} 
\limsup_{n\to +\infty}\tilde{\mathcal{F}}_{0}(u_n,h_n)\leq \tilde{\mathcal{F}}(u,h).
\ee
We subdivide the argument into two steps.\\

\noindent\textbf{Step 1.} 
In this step we prove that there exists a sequence $\{(u_n,h_n)\}\subset X_{\rm Lip}$ such that 
\begin{align}
&(u,h_n)\to (u,h)\quad\text{in }X,\label{sequence_convergence}\\
&|\Omega_{h_n}^+|=|\Omega_h^+|, \label{volume_constraint}\\
&\textrm{and}\,\,\,\,\lim_{n\to +\infty}\tilde{\mathcal{F}}(u_n,h_n)=\tilde{\mathcal{F}}(u,h). \label{eq:approx-lip}
\end{align}

We begin by observing that the construction introduced in \cite[Proof of Theorem 2.8, Steps 3--5]{FFLM2} yields a sequence $\{(u,\tilde{h}_n)\}\subset X_{\rm Lip}$ such that 
\begin{align*}
&0\leq \tilde{h}_n(x)\leq h(x)\quad\text{for every $x\in[a,b]$,}\\
&\tilde{h}_n\to h\quad\text{pointwise in $[a,b]$,}\\
&(u,\tilde{h}_n)\to (u,h)\quad\text{in }X,\\
\end{align*}
and
\be{surface_energy_convergence2}\lim_{n\to +\infty}\int_{\Gamma_{\tilde{h}_n}}\varphi(y)\,d\HH^1 = \int_{\tilde{\Gamma}_h}\varphi(y)\,d\HH^1\,+\,2\gamma_{\rm f} \HH^1(\Gamma_h^{cut}).\ee

The remaining part of this step is devoted to \MMM modify \EEE the sequence $\tilde{h}_n$ in order to obtain a sequence $h_n$ not only satisfying \eqref{sequence_convergence} and \eqref{eq:approx-lip}, but also the volume constraint \eqref{volume_constraint}. With this aim, let us measure how much the volume associated to each $\tilde{h}_n$ differs from the one of $h$ by a parameter $\lambda_n$ defined \MMM as \EEE
\begin{equation}\label{lambdan} 
\lambda_n:=\left(|\Omega_h^+|- |\Omega^+_{\tilde{h}_n}| \right)^{r}\geq0
\end{equation}
for every $n\in\mathbb{N}$ and for a fixed number $r\in(0,1)$. \MMM For every $n\in \mathbb{N}$, let \EEE $h_n :(a,b) \to \mathbb{R}_+$ be the function \MMM given \EEE by
\begin{equation}\label{hndefinition} 
h_n(x):=
\begin{cases}
\tilde{h}_n(x) &\textrm{if $\tilde{h}_n(x)=0$},\\
\tilde{h}_n(x)+\varepsilon_n &\textrm{if $\tilde{h}_n(x)\geq\lambda_n$},\\
\left(1+\frac{\varepsilon_n}{\lambda_n}\right)\tilde{h}_n(x)&\textrm{if $\tilde{h}_n(x)\in(0,\lambda_n)$}
\end{cases}
\end{equation}
for every $x\in[a,b]$ and for 
\begin{equation}\label{epsilonn} 
\varepsilon_n:=\frac{1}{\mu_n}\left(|\Omega_h^+|- |\Omega^+_{\tilde{h}_n}|\right),
\end{equation}
where $\mu_n$ is given by
$$
\mu_n:=|\tilde{H}_{\lambda_n}|\,+\,\frac{1}{\lambda_n}\int_{[a,b]\setminus \tilde{H}_{\lambda_n}}\tilde{h}_n(x)\,\mathrm{d}x
$$
with $\tilde{H}_{\lambda_n}:=\{x\in[a,b]\,:\, \tilde{h}_n(x)\geq\lambda_n\}$. \MMM Note that, by contruction, $|\Omega^+_{h_n}|=|\Omega_h^+|$. \EEE Since $\lambda_n\to0$, by the $L^1$-convergence of $\{\tilde{h}_n\}$, we can apply Lemma \ref{Vitali_lemma} and obtain a \MMM constant \EEE $\mu>0$ and \MMM a corresponding \EEE integer $N_{\mu}$ such that 
$$
\mu_n>\mu
$$
for every $n\geq N_{\mu}$. Then, from \eqref{epsilonn} \MMM we obtain \EEE
\begin{equation}\label{epsilonnzero}
0\leq\varepsilon_n\leq\frac{1}{\mu}\left(|\Omega_h^+|- |\Omega^+_{\tilde{h}_n}|)\right)\to0,
\end{equation}
and 
\begin{equation}\label{ratiozero}
0\leq\frac{\varepsilon_n}{\lambda_n}\leq\frac{1}{\mu}\left(|\Omega_h^+|- |\Omega^+_{\tilde{h}_n}|\right)^{1-r}\to0
\end{equation}
since $r\in(0,1)$. Note that \eqref{sequence_convergence} together with \eqref{epsilonnzero}  and the fact that $\tilde{h}_n \leq h_n \leq \tilde{h}_n +\varepsilon_n$   implies that
\begin{equation}\label{first_assert_half}
\Rz^2 \setminus \Omega_{h_n} \to \Rz^2 \setminus \Omega_h
\end{equation}
 with respect to the Hausdorff-distance. Furthermore, by also employing Bolzano's Theorem \MMM we deduce \EEE
$$
 |h_n(x)-h_n(x')|\leq C_n\left(1+\frac{\varepsilon_n}{\lambda_n}\right)|x-x'| 
$$
for every $x, x'\in[a,b]$, where $C_n>0$ denotes the Lipschitz constant associated to $\tilde{h}_n$. Hence, \MMM the maps \EEE $h_n$ are also Lipschitz. We now prove that 
\begin{equation}\label{surface_energy_convergence} 
\int_{\Gamma_{h_n}}\varphi(y)\,d\HH^1 \to \int_{\tilde{\Gamma}_h}\varphi(y)\,d\HH^1\,+\,2\gamma_{\rm f} \HH^1(\Gamma_h^{cut}).
\end{equation}
By  the definition of $h_n$ we have that 
\begin{align}
\Big |\int_{\Gamma_{\tilde{h}_n}}{\varphi}(y)\,d\HH^1 &- \int_{\Gamma_{h_n}}{\varphi}(y)\,d\HH^1\Big |\nonumber\\
&= \gamma_f\left|\mathcal{H}^1(\Gamma_{\tilde{h}_n}\cap \{\lambda_n+\varepsilon_n>y>0\})-\mathcal{H}^1(\Gamma_{h_n}\cap \{\lambda_n>y>0\})\right|\nonumber\\
&= \gamma_f\sum_{i\in\mathcal{I}^n}\int_{a^n_i}^{b^n_i} \left|\sqrt{1+(\tilde{h}'_n)^2}-\sqrt{1+(h'_n)^2}\right| \,\mathrm{d}x
\label{surface_energy_difference}
\end{align}
for some index set $\mathcal{I}^n$, and \MMM for points \EEE $a_i^n<b_i^n$ with $(a_i^n,b_i^n) \cap (a_j^n,b_j^n)= \emptyset$ for all $i,j \in \mathcal{I}^n, i \neq j$  such that
$$
\bigcup_{i\in\mathcal{I}^n}(a_i^n,b_i^n)= \{x\in[a,b]\,:\, 0<\tilde{h}_n(x)<\lambda_n\}.
$$
We now observe that \MMM on each interval $(a_i^n,b_i^n)$ there holds \EEE
\begin{align}
\left| \sqrt{1+(\tilde{h}'_n)^2}-\sqrt{1+(h'_n)^2}\right|&\leq  \frac{\sigma_n(\tilde{h}'_n)^2}{\sqrt{1+(\tilde{h}'_n)^2}+\sqrt{1+(h'_n)^2}}\nonumber\\
&\leq\frac{\sigma_n(\tilde{h}'_n)^2}{2\sqrt{1+({h}'_n)^2}}\nonumber\\
&\leq\frac{\sigma_n}{2}\sqrt{1+({h}'_n)^2}
 \label{An}
\end{align}
where  the parameter 
$$\sigma_n:=\left[\left(\frac{\varepsilon_n}{\lambda_n}\right)^2+2\frac{\varepsilon_n}{\lambda_n}\right]$$
is such that 
\be{sigmanzero} \sigma_n\to 0
\ee
by \eqref{ratiozero}.
Therefore, by combining \eqref{surface_energy_difference} with \eqref{An}  we obtain
\begin{align}
\Big |\int_{\Gamma_{\tilde{h}_n}}{\varphi}(y)\,d\HH^1 &- \int_{\Gamma_{h_n}}{\varphi}(y)\,d\HH^1\Big|\nonumber\\
& \leq \frac{\gamma_f\sigma_n}{2}  \sum_{i\in\mathcal{I}^n}\int_{a^n_i}^{b^n_i}   \sqrt{1+({h}'_n)^2}   \,\mathrm{d}x_1\nonumber\\
& \leq  \frac{\gamma_f\sigma_n}{2}   \mathcal{H}^1(\Gamma_{{h}_n})\to 0 \label{surface_energy_convergence3}
\end{align}
where we used \eqref{sigmanzero}  and the fact that  by \eqref{surface_energy_convergence2} there exists a constant $C>0$ for which 
\be{graphsbounded}
\mathcal{H}^1(\Gamma_n)\leq C
\ee
for every $n\in\N$. From \eqref{surface_energy_convergence2} and \eqref{surface_energy_convergence3}  \MMM we deduce \EEE \eqref{surface_energy_convergence}.  

Let us now define $u_n: \Omega_{h_n}\to\mathbb{R}^2$ by
\begin{equation}\label{undefinition} 
u_n(x,y):=
\begin{cases}
u(x,y-\varepsilon_n) &\textrm{if $y>y^0+\varepsilon_n$},\\
u(x,y^0) &\textrm{if $y^0+\varepsilon_n\geq y>y^0$},\\
u(x,y) &\textrm{if $y^0\geq y$,}
\end{cases}
\end{equation}
\MMM where $y^0<-\ep_n$ is \EEE chosen in such a way that $u(\cdot,y^0)\in H^1((a,b);\mathbb{R}^2)$. Note that \MMM the maps \EEE $u_n$ are well defined in $\Omega_{h_n}$ since $h_n\leq \tilde{h}_n+\varepsilon_n\leq h+\varepsilon_n$. Furthermore, by  \eqref{epsilonnzero} and \eqref{undefinition} we have that $u_n\rightharpoonup u$ in $H^1_{\rm loc}(\Omega';\mathbb{R}^2)$ for every $\Omega'\subset\subset\Omega_h$ as $n\to+\infty$, which together with \eqref{first_assert_half} and \eqref{graphsbounded} yields \eqref{sequence_convergence}. 

In the remaining part of this step  we prove \MMM that \EEE
\begin{equation}\label{elastic_energy_convergence} 
\int_{\Omega_{h_n}}W_0(y, Eu_n(x,y)-E_0(y))\,dx\,dy\to \int_{\Omega_h}W_0(y, Eu(x,y)-E_0(y))\,dx\,dy
\end{equation}
which together with \eqref{surface_energy_convergence} implies \eqref{eq:approx-lip}. We begin by observing that 
 \begin{align}\label{elastic_energy_convergence_novolume}
\int_{\Omega_{\tilde{h}_n}}W_0(y, Eu(x,y)-E_0(y))\,dx\,dy\to \int_{\Omega_h}W_0(y, Eu(x,y)-E_0(y))\,dx\,dy
\end{align}
by the Monotone Convergence Theorem. Furthermore, by \eqref{undefinition} \MMM there holds \EEE
 \begin{align}
 \Big|\int_{\Omega_{h_n}}W_0(y, &Eu_n(x,y)-E_0(y))\,dx\,dy- \int_{\Omega_{\tilde{h}_n}}W_0(y, Eu(x,y)-E_0(y))\,dx\,dy\Big| \nonumber\\
  \leq &\,C\Big[\,\int_{[a,b]\times[y^0,y^0+\varepsilon_n]} (W_0(y,Eu(x,y^0))\MMM+W_0(y,Eu(x,y)))\EEE\,dx\,dy + \int_{\{\tilde{h}_n>0\}\times[0,\varepsilon_n]}  |E_0(y)|^2   \,dx\,dy\nonumber\\
  &
\quad\, + \int_{\{\tilde{h}_n=0\}\times[-\varepsilon_n,0]} W_0(y,Eu(x,y))\,dx\,dy \,+ \, \int_{E_n} W_0(y,Eu(x,y)-E_0(y))\,dx\,dy  \,\Big]\nonumber\\
   \leq \,&\, C\Big[\,\varepsilon_n\int_{[a,b]} |Eu(x,y^0)|^2\,dx
  +  \int_{[a,b]\times[0,\varepsilon_n]}  |E_0(y)|^2   \,dx\,dy  \nonumber\\
  & \quad\,+  \int_{[a,b]\times([-\varepsilon_n,0]\MMM\cup [y^0,y^0+\ep_n])\EEE}  |Eu|^2 \,dx\,dy  
  +  \int_{E_n} (|Eu|^2+|E_0(y)|^2) \,dx\,dy   \,\Big] \label{elastic_energy_convergence2} 
 \end{align}
 where $\{\tilde{h}_n=0\}:=\{x\in[a,b]\,:\, \tilde{h}_n(x)=0\}$, $\{\tilde{h}_n>0\}:=[a,b]\setminus\{\tilde{h}_n=0\}$, $\{\lambda_n>\tilde{h}_n>0\}:=\{x\in[a,b]\,:\, \lambda_n>\tilde{h}_n(x)>0\}$, and 
 $$
 E_n:=\left(\Omega_{\tilde{h}_n}\setminus(\Omega_{h_n}-\varepsilon_n\textbf{e}_2)\right)\cap\left(\{\lambda_n>\tilde{h}_n>0\}\times(0,+\infty)\right).$$
Notice that $|E_n|\leq C\varepsilon_n$ for some constant $C>0$, since  $0<\tilde{h}_n-(h_n-\varepsilon_n)\leq \varepsilon_n$ for every $x\in\{\lambda_n>\tilde{h}_n>0\}$ by \eqref{hndefinition}. Therefore,  from  \eqref{epsilonnzero} and \eqref{elastic_energy_convergence2} \MMM we conclude \EEE that 
$$
 \Big|\int_{\Omega_{h_n}}W_0(y, Eu_n(x,y)-E_0(y))\,dx\,dy- \int_{\Omega_{\tilde{h}_n}}W_0(y, Eu(x,y)-E_0(y))\,dx\,dy\Big| \to0
$$
as $n\to+\infty$, and hence, also in view of \eqref{elastic_energy_convergence_novolume}, we obtain \eqref{elastic_energy_convergence}.


\noindent\textbf{Step 2.}  In the case in which $\gamma_s-\gamma_{fs}\leq\gamma_f$, there holds 
$$\tilde{\mathcal{F}}(u_n,h_n)=\tilde{\mathcal{F}}_0(u_n,h_n)$$
and hence, \eqref{limsup} directly  follows  from \eqref{eq:approx-lip}. Therefore, the sequence constructed in Step 1 realizes \eqref{condition_1}, \eqref{condition_2}, and \eqref{limsup} for the case $\gamma_s-\gamma_{fs}\leq\gamma_f$. It remains to treat the case $\gamma_f<\gamma_s-\gamma_{fs}$. 
 In view of the previous step, \EEE and by a diagonal argument, the thesis reduces to show that for every $(\bar{u},\bar{h})\in X_{\rm Lip}$ there exists a sequence $\{(\bar{u}_n,\bar{h}_n)\}\subset X_{\rm Lip}$ such that 
\begin{align*}
&(\bar{u}_n,\bar{h}_n)\to (\bar{u},\bar{h})\quad\text{in }X,\\
&|\Omega_{\bar{h}_n}^+|=|\Omega_h^+|,
\end{align*}
and
\be{eq:interm-conv}\lim_{n\to +\infty}\tilde{\mathcal{F}}_0(\bar{u}_n,\bar{h}_n)= \tilde{\mathcal{F}}(\bar{u},\bar{h})  \ee

Fix $(\bar{u},\bar{h})\in X_{\rm Lip}$. 
 We define $\bar{h}_n$ by 
$$\bar{h}_n(x):=\min\{\bar{h}(x)+\ep_n,t_n\}$$
for every $x\in [a,b]$, where $\{\ep_n\}$ is a vanishing sequence of positive numbers, and $\{t_n\}$ is chosen so that $t_n>0$ and $|\Omega_{\bar{h}_n}^+|=|\Omega_{\bar{h}}^+|$ for every $n\in \N$.
Choosing $y_0<0$ such that $u(\cdot,y_0)\in H^1((a,b);\R^2)$ (the existence of $y_0$ follows by a slicing argument), we set,
$$\bar{u}_n(x,y):=\begin{cases}\bar{u}(x,y-\ep_n)&\text{if }y>y_0+\ep_n,\\
\bar{u}(x,y_0)&\text{if }y_0-\ep_n\leq y\leq y_0+\ep_n,\\
\bar{u}(x,y)&\text{if }y<y_0-\ep_n,\end{cases}$$
for every $(x,y)\in \Omega_{\bar{h}_n}$.
By definition,
$$\lim_{n\to +\infty}\int_{\Omega_{\bar{h}_n}}W_0(y, E\bar{u}_n(x,y)-E_0(y))\,dx\,dy= \int_{\Omega_{\bar{h}}}W_0(y, Eu(x,y)-E_0(y))\,dx\,dy,$$
and
 \be{eq:ptw}\varphi_0(\min\{y+\ep_n,t_n\})=\gamma_f=\varphi(y)\ee 
for every $y\geq 0$. Property \eqref{eq:interm-conv} follows then by the observation that
$$\limsup_{n\to +\infty}\int_{\Gamma_{\bar{h}_n}}\varphi_0(y)\,d\HH^1\leq \lim_{n\to +\infty} \int_{\Gamma_h}\varphi_0(\min\{y+\ep_n,t_n\})\,d\HH^1=\int_{\Gamma_h}\varphi(y)\,d\HH^1,$$
where in the last equality we used \eqref{eq:ptw}. 
\end{proof}

\subsection{$\Gamma$-convergence from the transition-layer model}

In this subsection we characterize $\mathcal{F}$  defined in \eqref{filmenergyfinal} as the  $\Gamma$-limit of the transition-layer functionals  $\mathcal{F}_{\delta}$. The proof of this result is a modification of the arguments in \cite[Theorems 2.8 and 2.9]{FFLM2} to the situation with possibly $\C_f\not=\C_s$ and $\gamma_{fs}\neq0$, therefore we here highlight only the main changes for convenience of the reader.

We begin by characterizing the lower-semicontinuous envelope of $\mathcal{F}_{\delta}$ with respect to the convergence in $X$,  restricted to pairs in $X_{\rm Lip}$, with the integral formula \eqref{eq:barFdeltaformula}.

\begin{proposition}[Relaxation of the transition-layer functionals]
\label{prop:relax-delta}
For every $\delta>0$, let $\bar{\mathcal{F}}_{\delta}$ be the relaxed functional of $\mathcal{F}_{\delta}$, namely
\begin{align}
\label{eq:barFdelta}
\bar{\mathcal{F}}_{\delta}(u,h):=\inf&\left\{\liminf_{n\to +\infty} \mathcal{F}_{\delta}(u_n,h_n):\,(u_n,h_n)\in X_{\rm Lip},\right.\nonumber\\
&\qquad\qquad\,(u_n,h_n)\to (u,h)\,\text{in }X,\,\text{and }|\Omega_{h_n}^+|=|\Omega_h^+| \bigg\}
\end{align}
for every $(u,h)\in X$.
Then
\begin{align}
\label{eq:barFdeltaformula}
\bar{\mathcal{F}}_{\delta}(u,h)&=\int_{\Omega_h}W_{\delta}(y, Eu(x,y)-E_0(y))\,dx\,dy+\int_{\tilde{\Gamma}_h}\varphi_{\delta}(y)\,d\HH^1\nonumber\\
&\qquad\qquad\qquad+2\sum_{x\in S}\int_{h(x)}^{h^-(x)}\varphi_{\delta}(y)\,dy+\gamma_{fs}(b-a),
\end{align}
for every $(u,h)\in X$.
\end{proposition}
\begin{proof} Denote by $\hat{\mathcal{F}}_{\delta}$ the right-hand side of \eqref{eq:barFdelta}. The proof of the inequality
$$\bar{\mathcal{F}}_{\delta}(u,h)\geq \hat{\mathcal{F}}_{\delta}(u,h)$$
for every $(u,h)\in X$ is analogous to \cite[Proof of Theorem 2.8, Step 1]{FFLM2}. To prove the opposite inequality, we argue as in \cite[Proof of Theorem 2.8, Steps 3--5]{FFLM2}, and we construct a sequence $\{h_n\}$ of Lipschitz maps such that
\begin{align}
\label{eq:conditions-hn}
&0\leq h_n(x)\leq h(x)\quad\text{for every }x\in [a,b],\\
&\nonumber(u,h_n)\to (u,h)\quad\text{in }X,\\
&\nonumber\lim_{n\to +\infty}\mathcal{F}_{\delta}(u,h_n)=\hat{\mathcal{F}}_{\delta}(u,h).
\end{align}
With a slicing argument we identify $y_0<0$ such that $u(\cdot,y_0)\in H^1((a,b);\R^2)$, and we define the maps
$$u_n(x,y):=\begin{cases}
u(x,y-\ep_n)&\text{if }y>y_0+\ep_n,\\
u(x,y_0)&\text{if }y_0<y\leq y_0+\ep_n,\\
u(x,y)&\text{if }y\leq y_0
\end{cases}$$
for a.e. $(x,y)\in \Omega_{\tilde{h}_n}$, where $\tilde{h}_n(x):=h_n(x)+\ep_n$ for every $x\in [a,b]$, and
$$\ep_n:=\frac{1}{b-a}\Big(|\Omega_h^+|-\int_a^b h_n(x)\,dx\Big).$$
It is immediate to see that $|\Omega_{\tilde{h}_n}^+|=|\Omega_h^+|$, and that $(u_n,\tilde{h}_n)\to (u,h)$ in $X$. Additionally,
\be{eq:serven1}
\lim_{n\to +\infty}\int_{\Gamma_{h_n}}\varphi_{\delta}(y)\,d\HH^1=\lim_{n\to +\infty}\int_{\Gamma_{\tilde{h}_n}}\varphi_{\delta}(y)\,d\HH^1.
\ee
Regarding the bulk energies, we have
\begin{align*}
&\int_{\Omega_{\tilde{h}_n}}W_{\delta}(y,Eu_n(x,y)-E_{\delta}(y))\,dx\,dy=\int_a^b\int_{-\infty}^{y_0}W_{\delta}(y,Eu(x,y)-E_{\delta}(y))\,dy\,dx\\
&\qquad\qquad\qquad\qquad\qquad\qquad+\int_a^b\int_{y_0}^{y_0+\ep_n}W_{\delta}(y,Eu(x,y_0)-E_{\delta}(y))\,dy\,dx\\
&\qquad\qquad\qquad\qquad\qquad\qquad+\int_a^b\int_{y_0+\ep_n}^{h_n(x)+\ep_n}W_{\delta}(y,Eu(x,y-\ep_n)-E_{\delta}(y))\,dy\,dx.
\end{align*}
Thus, by \eqref{eq:quadratic-form} there holds
\begin{align}
\label{eq:add-goes-to-zero}
&\Big|\int_{\Omega_{\tilde{h}_n}}W_{\delta}(y,Eu_n(x,y)-E_{\delta}(y))\,dx\,dy-\int_{\Omega_{{h}_n}}W_{\delta}(y,Eu(x,y)-E_{\delta}(y))\,dx\,dy\Big|\\
&\nonumber\quad\leq C\int_a^b\int_{y_0}^{y_0+\ep_n}|Eu(x,y)-E_{\delta}(y)|^2\,dy\,dx\\
&\nonumber\quad+\int_a^b\int_{y_0}^{h_n(x)}|W_{\delta}(y+\ep_n,Eu(x,y)-E_{\delta}(y+\ep_n))-W_{\delta}(y,Eu(x,y)-E_{\delta}(y))|\,dx\,dy\\
&\nonumber\quad\leq C\int_a^b\int_{y_0}^{y_0+\ep_n}|Eu(x,y)-E_{\delta}(y)|^2\,dy\,dx+C\int_a^b \int_{y_0}^{h_n(x)}|E_{\delta}(y+\ep_n)-E_{\delta}(y)|^2\,dy\,dx\\
&\nonumber\quad+C\int_a^b \int_{y_0}^{h_n(x)}(\C_{\delta}(y+\ep_n)-\C_{\delta}(y))(Eu(x,y)-E_{\delta}(y)):(Eu(x,y)-E_{\delta}(y))\,dy\,dx,
\end{align}
which converges to zero due to the Dominated Convergence Theorem. By combining \eqref{eq:conditions-hn}, \eqref{eq:serven1}, and \eqref{eq:add-goes-to-zero} we deduce that
$$\lim_{n\to+\infty}\mathcal{F}_{\delta}(u_n,\tilde{h}_n)=\lim_{n\to+\infty}\mathcal{F}_{\delta}(u,h_n)=\hat{\mathcal{F}}_{\delta}(u,h),$$
which in turn yields
$$\bar{\mathcal{F}}_{\delta}(u,h)\leq \hat{\mathcal{F}}_{\delta}(u,h)$$
and completes the proof of the proposition.
\end{proof}
Proposition \ref{prop:relax-delta} is instrumental for the proof of the $\Gamma$-convergence result.
\begin{proposition}[$\Gamma$-convergence]
\label{prop:gamma-conv}
The functional $\mathcal{F}$ is the $\Gamma$-limit as $\delta\to 0$ of $\{F_{\delta}\}_{\delta}$ under volume constraint. Namely, if $(u_{\delta},h_{\delta})\to (u,h)$ in $X$, and $|\Omega_{h_{\delta}}^+|=|\Omega_h^+|$ for every $\delta$, then
$$\mathcal{F}(u,h)\leq \liminf_{\delta\to 0}\mathcal{F}_{\delta}(u_{\delta},h_{\delta}).$$
Additionally, for every $(u,h)\in X$, there exists a sequence $\{(u_{\delta},h_{\delta})\}\subset X$ such that $|\Omega_{h_{\delta}}^+|=|\Omega_h^+|$ for every $\delta$, and
$$\mathcal{F}(u,h)\geq \limsup_{\delta\to 0}\mathcal{F}_{\delta}(u_{\delta},h_{\delta}).$$
\end{proposition}
\begin{proof}
We subdivide the proof into two steps.\\

\noindent\textbf{Step 1.}
We first show that for all sequences $\{\delta_n\}$, and $\{(u_n,h_n)\}\subset X_{\rm Lip}$, with $\delta_n\to 0$, $(u_n,h_n)\to (u,h)$ in $X$, and such that $|\Omega_{h_n}^+|=|\Omega_h^+|$ for every $n\in \N$, there holds
\be{eq:add-limsup}
\lim_{n\to +\infty}\mathcal{F}_{\delta_n}(u_n,h_n)\geq\mathcal{F}(u,h).
\ee
 The liminf inequality for the surface energies follows arguing as in \cite[Proof of Theorem 2.9, Step 1]{FFLM2}. To study the elastic energies fix $D\subset\subset\Omega_h$ and let $\eta>0$. Let $\ep>0$ be small enough so that
\be{eq:smallenough}
 \int_{D\cap\{|y|\leq \ep\}} W_{0}(y, Eu(x,y)-E_{0}(y))\,dx\,dy\leq \eta.
 \ee
 We have
 \begin{align*}
 &\liminf_{n\to +\infty}\int_{\Omega_{h_n}}W_{\delta_n}(y, Eu_n(x,y)-E_{\delta_n}(y))\,dx\,dy\\
 &\quad\geq \liminf_{n\to +\infty}\int_D W_{\delta_n}(y, Eu_n(x,y)-E_{\delta_n}(y))\,dx\,dy\\
 &\quad\geq \liminf_{n\to +\infty}\int_{D\cap\{|y|>\ep\}} W_{\delta_n}(y, Eu_n(x,y)-E_{\delta_n}(y))\,dx\,dy\\
 &\qquad+\liminf_{n\to +\infty}\int_{D\cap\{|y|\leq \ep\}} W_{\delta_n}(y, Eu_n(x,y)-E_{\delta_n}(y))\,dx\,dy.
 \end{align*}
 Now,
 \begin{align}
 \label{eq:serven2}
 &\int_{D\cap\{|y|>\ep\}} W_{\delta_n}(y, Eu_n(x,y)-E_{\delta_n}(y))\,dx\,dy\\
 &\quad\nonumber=\int_{D\cap\{|y|>\ep\}}(\C_{\delta_n}(y)-\C(y))(Eu_n(x,y)-E_{\delta_n}(y)):(Eu_n(x,y)-E_{\delta_n}(y))\,dx\,dy\\
 &\qquad\nonumber +\int_{D\cap\{|y|>\ep\}} W_{0}(y, Eu_n(x,y)-E_{\delta_n}(y))\,dx\,dy.
 \end{align}
 Since $(u_n,h_n)\to (u,h)$ in $X$, by Definition \ref{def:conv-X} the right-hand side of \eqref{eq:serven2} satisfies
 $$\liminf_{n\to +\infty}\int_{D\cap\{|y|>\ep\}} W_{0}(y, Eu_n(x,y)-E_{\delta_n}(y))\,dx\,dy\geq \int_{D\cap\{|y|>\ep\}} W_{0}(y, Eu(x,y)-E_{0}(y))\,dx\,dy,$$
 whereas the first term in the right-hand side of \eqref{eq:serven2} can be estimated as
 \begin{align*}
 &\Big|\int_{D\cap\{|y|>\ep\}}(\C_{\delta_n}(y)-\C(y))(Eu_n(x,y)-E_{\delta_n}(y)):(Eu_n(x,y)-E_{\delta_n}(y))\,dx\,dy\Big|\\
 &\quad\leq C\Big\|1-f\Big(\frac{y}{\delta_n}\Big)\Big\|_{L^{\infty}(D\cap\{y>\ep\})}+C\Big\|1+f\Big(\frac{y}{\delta_n}\Big)\Big\|_{L^{\infty}(D\cap\{y<-\ep\})},
 \end{align*}
 which converges to zero as $n\to +\infty$ due to the properties of $f$. Hence, by \eqref{eq:smallenough},
 \begin{align*}
 &\liminf_{n\to +\infty}\int_{\Omega_{h_n}}W_{\delta_n}(y, Eu_n(x,y)-E_{\delta_n}(y))\,dx\,dy\\
&\quad\geq\liminf_{n\to +\infty} \int_{D\cap\{|y|>\ep\}} W_{\delta_n}(y, Eu_n(x,y)-E_{\delta_n}(y))\,dx\,dy\\
&\quad\geq \int_{D\cap\{|y|>\ep\}} W_{0}(y, Eu(x,y)-E_{0}(y))\,dx\,dy\\
&\quad\geq \int_{D} W_{0}(y, Eu(x,y)-E_{0}(y))\,dx\,dy-\eta.
 \end{align*}
 By the arbitrariness of $\eta$ and $D$ we conclude that
  \begin{align*}
 &\liminf_{n\to +\infty}\int_{\Omega_{h_n}}W_{\delta_n}(y, Eu_n(x,y)-E_{\delta_n}(y))\,dx\,dy\\
 &\quad\geq \int_{\Omega_h} W_{0}(y, Eu(x,y)-E_{0}(y))\,dx\,dy.
 \end{align*}
\noindent\textbf{Step 2.} By Proposition \ref{prop:relax-delta} to prove the limsup inequality it is enough to show that for all sequences $\{\delta_n\}$ of nonnegative numbers, with $\delta_n\to 0$, and for every $(u,h)\in X$ there exists $\{(u_n,h_n)\}\subset X$ such that $(u_n,h_n)\to (u,h)$ in $X$, and
\be{eq:capire}\limsup_{n\to +\infty}\overline{\mathcal{F}}_{\delta_n}(u_n,h_n)\leq \mathcal{F}(u,h).\ee
Fix $\{\delta_n\}$. If $\gamma_f\geq \gamma_s-\gamma_{fs}$, take $u_n=u$ and $h_n=h$. Then \eqref{eq:capire} follows by the pointwise convergences 
$$\varphi_{\delta_n}(y)\to \varphi_0(y)\quad\text{for every }y\in [0,+\infty),$$
and
\be{eq:cdelta-point}\C_{\delta_n}(y)\to \C(y)\quad\text{for every }y\in \R.\ee
If $\gamma_f<\gamma_s-\gamma_{fs}$, construct $\ep_n\to 0$ such that
$$\varphi_{\delta_n}(y+\ep_n)\to\varphi_0(y)=\gamma_f\quad\text{for all }y\in [0,+\infty).$$
Let $y_0<0$ be such that $u(\cdot,y_0)\in H^1((a,b);\R^2)$. We define
$$u_n(x):=\begin{cases}
u(x,y-\ep_n)&\text{if }y>y_0+\ep_n,\\
u(x,y_0)&\text{if }y_0<y\leq y_0+\ep_n,\\
u(x,y)&\text{if }y\leq y_0,
\end{cases}$$
and $h_n(x):=\min\{h(x)+\ep_n,t_n\}$, where $t_n>0$ is such that $|\Omega_{h_n}^+|=d$. The convergence of surface energies follows as in \cite[Proof of theorem 2.9, Step 2]{FFLM2}. Regarding the bulk energies, we have 
\begin{align}
\label{eq:add-bulk}
&\int_{\Omega_{h_n}}W_{\delta_n}(y,Eu_n(x,y)-E_{\delta_n}(y))\,dx\,dy=\int_a^b\int_{-\infty}^{y_0}W_{\delta_n}(y,Eu(x,y)-E_{\delta_n}(y))\,dx\,dy\\
&\nonumber\quad+\int_a^b\int_{y_0}^{y_0+\ep_n}W_{\delta_n}(y,Eu(x,y_0)-E_{\delta_n}(y))\,dx\,dy\\
&\nonumber\quad+\int_a^b\int_{y_0+\ep_n}^{h_n(x)}W_{\delta_n}(y,Eu(x,y-\ep_n)-E_{\delta_n}(y))\,dx\,dy.
\end{align}
The first term in the right-hand side of \eqref{eq:add-bulk} satisfies
\begin{align}
\label{eq:bulk1}
\lim_{n\to +\infty}&\int_a^b\int_{-\infty}^{y_0}W_{\delta_n}(y,Eu(x,y)-E_{\delta_n}(y))\,dx\,dy\\
&\nonumber\quad=\int_a^b\int_{-\infty}^{y_0}W_{0}(y,Eu(x,y)-E_{0}(y))\,dx\,dy,
\end{align}
owing to \eqref{eq:cdelta-point} and the fact that 
\be{eq:edeltan}E_{\delta_n}\to E_0\quad\text{strongly in }L^2_{\rm loc}(\R;\M^{2\times 2}_{\rm sym}).\ee
By \eqref{eq:quadratic-form} the second term in the right-hand side of \eqref{eq:add-bulk} can be bounded from above as
\begin{align}
\label{eq:bulk2}
&\int_a^b\int_{y_0}^{y_0+\ep_n}W_{\delta_n}(y,Eu(x,y_0)-E_{\delta_n}(y))\,dx\,dy\\
&\nonumber\quad\leq C\int_a^b\int_{y_0}^{y_0+\ep_n}|Eu(x,y_0)-E_{\delta_n}(y)|^2\,dy\,dx
\end{align}
and hence vanishes, as $n\to +\infty$. Finally, there holds
\begin{align*}
&\int_a^b\int_{y_0+\ep_n}^{h_n(x)}W_{\delta_n}(y,Eu(x,y-\ep_n)-E_{\delta_n}(y))\,dx\,dy\\
&\quad\leq \int_{\Omega_h}W_{\delta_n}(y+\ep_n,Eu(x,y)-E_{\delta_n}(y+\ep_n))\,dx\,dy\\
&\quad=\int_{\Omega_h}(\C_{\delta_n}(y{+}\ep_n){-}\C(y))(Eu(x,y){-}E_{\delta_n}(y{+}\ep_n)){:}(Eu(x,y){-}E_{\delta_n}(y{+}\ep_n))\,dx\,dy\\
&\quad+\int_{\Omega_h}W_{0}(y,Eu(x,y)-E_{\delta_n}(y+\ep_n))\,dx\,dy.
\end{align*}
By the Dominated Convergence Theorem, \eqref{eq:cdelta-point}, and \eqref{eq:edeltan}, we conclude that
\begin{align}
\label{eq:bulk3}
&\limsup_{n\to +\infty}\int_a^b\int_{y_0+\ep_n}^{h_n(x)}W_{\delta_n}(y,Eu(x,y-\ep_n)-E_{\delta_n}(y))\,dx\,dy\\
&\nonumber\quad\leq\int_{\Omega_h}W_{0}(y,Eu(x,y)-E_{0}(y))\,dx\,dy.
\end{align}
Inequalities \eqref{eq:bulk1}--\eqref{eq:bulk3} imply the convergence of the elastic energies and complete the proof of \eqref{eq:capire}.
\end{proof}

\section{Properties of local minimizers}\label{sec:regularity}

In this section we present a first regularity result for $\mu$-local minimizers $(u,h)$ of \eqref{filmenergyfinal}. Employing an argument first introduced in \cite{CL}, we prove that optimal profiles $h$ satisfy the \emph{internal-ball condition}. 

\MMM In what follows, denote by $\beta$ the quantity
\be{eq:beta}
\beta:=\frac{\min\{\gamma_f,\gamma_s-\gamma_{fs}\}}{\gamma_f}.
\ee \EEE

In order to prove the internal-ball condition we need to perform local variations. Therefore, we first show that  the area constraint in the  minimization problem of Definition \ref{def:local-min} can be replaced with a suitable penalization in the energy functional. 

\begin{proposition}\label{prop:penalization}
Let $(u, h)\in X$ be a $\mu$-local minimizer for the functional $\mathcal{F}$. Then there exists $\lambda_0>0$ such that
\be{eq:penprob}
\mathcal{F}(u, h)=\min\left\{\mathcal{F}(v, g)+\lambda||\Omega_h^+|-|\Omega_g^+||:\,(v,g)\in X,\,|\Omega_g\Delta \Omega_h|\leq \frac{\mu}{2}\right\}
\ee
for all $\lambda\geq \lambda_0$.
\end{proposition}

\begin{proof}
The proof strategy is analogous to \cite[Proof of Proposition 3.1]{FFLM2}, and consists in first establishing the existence of a solution $(u_{\lambda},h_{\lambda})$ of the minimum problem ($M_{\lambda}$) in the right side of \eqref{eq:penprob} for every fixed $\lambda>0$, then in observing that
\be{eq:firstinequality}
\mathcal{F}(u, h)\geq\mathcal{F}(u_{\lambda}, h_{\lambda})
\ee
since $(u, h)$ is admissible for ($M_{\lambda}$), and finally in showing that there exists $\lambda_0>0$ such that the reverse inequality 
\be{eq:reverseinequality}
\mathcal{F}(u, h)\leq\mathcal{F}(u_{\lambda}, h_{\lambda})
\ee
holds true for every $\lambda\geq \lambda_0$. Since $(u,h)$ is a $\mu$-local minimizer of the functional $\mathcal{F}$, to prove \eqref{eq:reverseinequality}  it is enough to show that $|\Omega_{g_{\lambda}}^+|=|\Omega_{h}^+|$ for all $\lambda\geq \lambda_0$ .

The key modification in our setting consists in observing that any sequence $\{(u_k,h_k)\}\subset X$, for which there exists a constant $C$ such that 
$$
\sup_{k\in \N}\mathcal{F}(u_k,h_k)<C,
$$
satisfies the uniform bound 
\be{eq:boundh}
\mathcal{H}^1(\Gamma_{h_k})\leq M(C),
\ee
where the constant $M(C)>0$ is given by 
$$
M(C):=\begin{cases}
\frac{C}{\beta\gamma_f} &\textrm{if $\beta\not=0$,}\\
\frac{C}{\min\{\gamma_f, \gamma_{fs}\}}&\textrm{if $\beta=0$,}
\end{cases}
$$
and where $\beta$ is the quantity defined in \eqref{eq:beta}. Note that, by \eqref{eq:ass-gammas}, when $\beta=0$ then $\gamma_{fs}=\gamma_s>0$. The bound \eqref{eq:boundh} is used a first time to prove the sequential compactness of any minimizing sequence for the problem ($M_{\lambda}$), and to deduce the existence of a minimizer $(u_{\lambda},h_{\lambda})$. In view of \eqref{eq:firstinequality}, an application of \eqref{eq:boundh} to the sequence $\{(u_{\lambda}, h_{\lambda})\}$ with $C=\mathcal{F}(u,h)$ allows to check that $\{(u_{\lambda}, h_{\lambda})\}$ satisfies the assumptions of \cite[Lemma 3.2]{FFLM2}, and to complete the proof of \eqref{eq:reverseinequality}.

\end{proof}

We are now ready to establish the internal-ball condition for optimal profiles.

\begin{proposition}[Internal-ball condition]
\label{thm:palla-interna}
Let $(u,h)\in X$ be a $\mu$-local minimizer for the functional $\mathcal{F}$. Then, there exists $\rho_0>0$ such that for every $z\in \overline{\Gamma}_h$ we can choose a point $P_z$ for which  $B(P_z,\rho_0)\cap((a,b)\times \R)\subset \Omega_h$, and 
$$\partial B(P_z,\rho_0)\cap\overline{\Gamma}_h=\{z\}.$$
\end{proposition}
\begin{proof}
Let $\lambda_0$ be as in Proposition \ref{prop:penalization} and let $\beta$ be the quantity defined in  \eqref{eq:beta}. The case in which $\beta=1$ can be treated as in \cite[Proposition 3.3]{FFLM2}, despite the fact that in our setting the two elasticity tensors $\C_f$ and $\C_s$ are allowed to be different. Also in the case $\beta<1$ the argument of \cite[Proposition 3.3]{FFLM2} can be implemented. We highlight the main differences with respect to the case $\beta=1$ for convenience of the reader.

We begin by proving the following claim: there exists $\rho_0>0$ such that, for any $P\in \R^2$ for which $B(P,\rho_0)\cap((a,b)\times \R)\subset \Omega_h$, the intersection between $\partial B(P,\rho_0)$ and $\overline{\Gamma}_h$ contains at most one point. Once this claim is proved, the uniform internal-ball condition of the assert follows then by the argument of \cite[Lemma 2]{CL}.

By contradiction, assume that for every $r>0$ there exists $\rho_r<\tfrac{r}{2}$ for which three points $P_1^r$, $P_2^r$, and $P_r$ can be chosen so that 
$$B(P_r,\rho_r)\cap((a,b)\times \R)\subset \Omega_h$$
and 
$$\partial B(P_r,\rho_r)\cap\overline{\Gamma}_h\supset\{P_1^r,P_2^r\}.$$
Denote by $[P_1^r,P_2^r]$ the segment  
$$[P_1^r,P_2^r]:=\{P_1^r+t(P_2^r-P_1^r)\,:\, 0\leq t\leq1\}$$ and by $\Gamma_{P_1^r,P_2^r}$ the set
$$
\Gamma_{P_1^r,P_2^r}:=\left(\tilde{\Gamma}_h\cap([x_1^r,x_2^r]\times\R)\right)\cup\displaystyle\bigcup_{i=1}^2\big\{(x_i^r,y)\,:\, y_i^r\leq y\leq h^+(x_i^r)\big\}
$$
where $P_1^r=:(x_1^r,y_1^r)$ and $P_2^r=:(x_2^r,y_2^r)$.  The case in which either $y_1^r\neq 0$ or $y_2^r\neq 0$ follows exactly as in \cite[Proposition 3.3]{FFLM2}, thus we assume that $y_1^r=y_2^r=0$. Consider the pair $(u,h^r)\in X$ with $h^r$ defined by
$$
h^r:=\begin{cases}
\displaystyle0&\textrm{if $x_1^r<x<x_2^r$},\\
h(x) &\textrm{otherwise}.
\end{cases}
$$
Note that $\Omega_h\setminus\Omega_{h^r}\subset \overline{D}^r$ where $D^r$ is the portion of $\R^+$ enclosed by the curve $\Gamma_{P_1^r,P_2^r}\cup[P_1^r,P_2^r]$. 

Fix
\be{eq:choice-ep0}
0<\ep_0<\frac{\mu}{4(b-a)},
\ee
and  
\be{eq:choice-ep}
0<\ep<\frac{\ep_0}{2},
\ee
and consider the finite set $A\subset (a,b)$ such that 
$$
\sum_{x\in J(h)\setminus A} (h^+(x)-h^-(x))+\sum_{x\in C(h)\setminus A} (h^-(x)-h(x))<\frac{\ep}{2}
$$
(see \cite[(3.33) and (3.34)]{FFLM2}). Let $r_0>0$ be such that
$$
r_0<\min\{|x-x'|\,:\, \textrm{$x\not=x'$ for any $x$, $x'\in A$}\}
$$
and
$$
\sup\{\mathscr{M}(I\setminus A)\,:\, \textrm{$I\in\mathcal{I}$}\}<\frac{\ep}{2}
$$
where $\mathcal{I}$ is the family of intervals $I\subset(a,b)$ with $|I|\leq r_0$, and $\mathscr{M}$ is the measure obtained by projecting $\mathcal{H}^1_{|_{\Gamma_h}}$ on the $x$-axis. By choosing 
\be{eq:choice-r}
r:=\min\left\{\frac{\ep_0}{4}, \frac{r_0}{2}\right\},
\ee
it follows that the set $[P_1^r,P_2^r]\cap A$ contains at most one point. Arguing as in \cite[Proof of (3.37)]{FFLM2} we deduce the estimate
\be{eq:estimateL}
\mathcal{H}^1(\Gamma_{P_1^r,P_2^r})\leq 2\ep+r.
\ee
In view of \eqref{eq:choice-ep}, \eqref{eq:choice-r}, and \eqref{eq:estimateL}, there holds
\begin{align*}
h^+(x)-h^r(x)&\leq \mathcal{H}^1(\Gamma_{P_1^r,P_2^r})+\mathcal{H}^1([P_1^r,P_2^r])\\
&\leq 2\ep+r+2\rho\leq 2\ep+2r\leq2\ep_0
\end{align*}
for every $x\in (x_1^r,x_2^r)$, and hence
\be{eq:meas-Dr}
|D^r|=\int_{x_1^r}^{x_2^r}(h^+(x)-h^r(x))\,dx\leq 2\ep_0(b-a),
\ee
and by \eqref{eq:choice-ep0},
$$|\Omega_h^+\Delta\Omega_{h^r}^+|\leq\frac{\mu}{2},$$
namely $(u,h^r)$ is an admissible competitor for the minimum problem \eqref{eq:penprob} with \mbox{$\lambda=\lambda_0$}. The minimality of $(u,h)$ yields the estimate
\be{eq:estimate1-int-ball}
\mathcal{F}(u,h)\leq\mathcal{F}(u,h^r)+\lambda_0||\Omega_{h^r}^+|-|\Omega_h^+||.
\ee
On the other hand,
\begin{align}
\label{eq:estimate2-int-ball}
\mathcal{F}(u,h^r)&=\int_{\Omega_{h^r}}W_0(y, Eu(x,y)-E_0(y))\,dx\,dy+\int_{\tilde{\Gamma}_{h^r}}\varphi(y)\,d\HH^1\\
&\nonumber\quad+2\gamma_f\HH^1(\Gamma_{h^r}^{cut})+\gamma_{fs}(b-a)\\
&\nonumber\leq \int_{\Omega_{h}}W_0(y, Eu(x,y)-E_0(y))\,dx\,dy{+}\gamma_f\big(\HH^1(\tilde{\Gamma}_h\cap\{y>0\}){-}\HH^1(\Gamma_{P_1^r,P_2^r})\big),\\
&\nonumber\quad+\min\{\gamma_f,\gamma_s-\gamma_{fs}\}\big(\HH^1(\tilde{\Gamma}_h\cap\{y=0\})+\HH^1([P_1^r,P_2^r])\big)\\
&\nonumber\quad+2\gamma_f\HH^1(\Gamma_{h}^{cut})+\gamma_{fs}(b-a)\\
&\nonumber=\mathcal{F}(u,h)-\gamma_f\big(\HH^1(\Gamma_{P_1^r,P_2^r})-\beta\HH^1([P_1^r,P_2^r])\big),
\end{align}
where $\beta$ is the quantity defined in \eqref{eq:beta}. By combining \eqref{eq:estimate1-int-ball} and \eqref{eq:estimate2-int-ball} we deduce that
\be{eq:isoperimetricopposite}
\HH^1(\Gamma_{P_1^r,P_2^r})-\beta\HH^1([P_1^r,P_2^r])\leq\frac{\lambda_0}{\gamma_f}|D^r|.
\ee
Arguing as in the proof of \cite[Lemma 1]{CL}, the isoperimetric inequality in the plane (see \cite{AFP}) yields
\be{eq:isoperimetric}
\sqrt{|D^r|}\leq\frac{\mathcal{H}^1(\partial D^r)}{2\sqrt{\pi}}=\frac{(\theta^r +1)\mathcal{H}^1([P_1^r,P_2^r])}{2\sqrt{\pi}}
\ee
where 
\be{eq:theta}
\theta^r:=\frac{\mathcal{H}^1(\Gamma_{P_1^r,P_2^r})}{\mathcal{H}^1([P_1^r,P_2^r])}>1.
\ee
Substituting \eqref{eq:isoperimetricopposite} in \eqref{eq:isoperimetric} we obtain the estimate
$$
|D^r|\leq\frac{\lambda_0^2}{4\pi\gamma_f^2}\frac{(\theta^r+1)^2}{(\theta^r-\beta)^2}|D^r|^2.
$$
In view of \eqref{eq:meas-Dr},
$$|D^r|\leq\frac{2\lambda_0^2\ep_0(b-a)}{4\pi\gamma_f^2}\frac{(\theta^r+1)^2}{(\theta^r-\beta)^2}|D^r|,$$
which in turn implies
$$\frac{2\lambda_0^2\ep_0(b-a)}{4\pi\gamma_f^2}\frac{(\theta^r+1)^2}{(\theta^r-\beta)^2}\geq 1.$$
By the previous inequality, as $\ep_0$ vanishes, then $\theta^r$ must approach $\beta$. Since $\beta<1$, we have a contradiction with \eqref{eq:theta}. This completes the proof of the claim and of the proposition.

\end{proof}

\noindent We notice that in view of Proposition \ref{thm:palla-interna}  the upper-end point of each \emph{cut} is a cusp point (see Figure \ref{graphs}). 

The following proposition is a consequence of the internal-ball condition.

\begin{proposition}
\label{prop:loc-lip}
Let $(u,h)\in X$ be a $\mu$-local minimizer for the functional $\mathcal{F}$. Then for any $z_0\in \overline{\Gamma}_h$ there exist an orthonormal basis $\mathbf{v}_1,\mathbf{v}_2\in\R^2$, and a rectangle $$Q:=\left\{z_0+s\mathbf{v}_1+t\mathbf{v}_2:\,-a'<s<a',\,-b'<t<b'\right\},$$
$a',b'>0$, such that $\Omega_h\cap Q$ has one of the following two representations:
\begin{itemize}
\item[1.] There exists a Lipschitz function $g:(-a',a')\to(-b',b')$ such that $g(0)=0$ and $$\qquad\Omega_h\cap Q:=\left\{z_0+s\mathbf{v}_1+t\mathbf{v}_2:\,-a'<s<a',\,-b'<t<g(s)\right\}\cap ((a,b)\times \R).$$
In addition, the function $g$ admits left and right derivatives at all points that are, respectively, left and right continuous.
\item[2.] There exist two Lipschitz functions $g_1,g_2:[0,a')\to(-b',b')$ such that $g_i(0)=(g_i)'_+(0)=0$ for $i=1,2$, $g_1\leq g_2$, and
               $$\quad\qquad\Omega_h\cap Q:=\left\{z_0+s\mathbf{v}_1+t\mathbf{v}_2:\,0<s<a',\,-b'<t<g_1(s)\text{ or }g_2(s)<t<b'\right\}.$$
In addition, the functions $g_1,g_2$ admit left and right derivatives at all points that are, respectively, left and right continuous.
\end{itemize}
\end{proposition}

For the proof of Proposition \ref{prop:loc-lip} we refer the reader to  \cite[Lemma 3]{CL} and \cite[Proposition 3.5]{FFLM2}. In particular Proposition \ref{prop:loc-lip} entails that the set 
$$\Gamma_h^{reg}=\Gamma_h\setminus(\Gamma_h^{cusp}\cup\Gamma_h^{cut})$$
is locally Lipschitz.

\section*{Acknowledgements}

The authors thank the Center for Nonlinear Analysis (NSF Grant No. DMS-0635983) and the Erwin Schr\"odinger Institute (Thematic Program: Nonlinear Flows), where part of this research was carried out. P. Piovano acknowledges support from the Austrian Science Fund (FWF) project P~29681 and the fact that this work has been funded by the Vienna Science and Technology Fund (WWTF), the City of Vienna, and Berndorf Privatstiftung through Project MA16-005. E. Davoli acknowledges the support of the Austrian Science Fund (FWF) project P~27052 and of the SFB project F65 ``Taming complexity in partial differential systems".

\end{document}